\documentclass[11pt,requationo]{amsart}
\usepackage[colorlinks=true,allcolors=blue,backref=page]{hyperref}
\usepackage{color}
\usepackage{amsmath, amssymb, amsthm}

\usepackage{mathrsfs}
\usepackage{mathtools}
\usepackage[noabbrev,capitalize,nameinlink]{cleveref}
\usepackage{aliascnt}
\usepackage{fullpage}
\usepackage[noadjust]{cite}
\usepackage{graphics}
\usepackage{pifont}
\usepackage{tikz}
\usepackage{tikz-cd}
\usepackage{bbm}
\usepackage[T1]{fontenc}

\usetikzlibrary{arrows.meta}

\usepackage{environ}
\usepackage{framed}
\usepackage{url}
\usepackage[linesnumbered,ruled,vlined]{algorithm2e}
\usepackage[noend]{algpseudocode}
\usepackage[labelfont=bf]{caption}
\usepackage{cite}
\usepackage{framed}
\usepackage[framemethod=tikz]{mdframed}
\usepackage{appendix}
\usepackage{graphicx}
\usepackage[textsize=tiny]{todonotes}
\usepackage{tcolorbox}
\usepackage{enumerate}
\usepackage[shortlabels]{enumitem}
\usepackage{physics}
\allowdisplaybreaks[1]

\apptocmd{\sloppy}{\hbadness 10000\relax}{}{} 

\crefname{equation}{}{}
\crefname{algocf}{Algorithm}{Algorithms}
\crefname{equation}{}{} 
\crefname{conjecture}{Conjecture}{Conjectures} 
\AtBeginEnvironment{appendices}{\crefalias{section}{appendix}} 

\crefformat{enumi}{#2#1#3}
\crefrangeformat{enumi}{#3#1#4 to~#5#2#6}
\crefmultiformat{enumi}{#2#1#3}%
{ and~#2#1#3}{, #2#1#3}{ and~#2#1#3}

\usepackage[color,final]{showkeys} 

\colorlet{refkey}{orange!20}
\colorlet{labelkey}{blue!30}

\crefname{algocf}{Algorithm}{Algorithms}

\numberwithin{equation}{section}
\newtheorem{theorem}{Theorem}[section]

\newaliascnt{proposition}{theorem}

\aliascntresetthe{proposition}
\crefname{proposition}{Proposition}{Propositions}

\newaliascnt{lemma}{theorem}

\aliascntresetthe{lemma}
\crefname{lemma}{Lemma}{Lemmas}

\newaliascnt{corollary}{theorem}

\aliascntresetthe{corollary}
\crefname{corollary}{Corollary}{Corollaries}

\crefname{subsubsection}{Step}{Steps}
\newtheorem{claim}[theorem]{Claim}
\crefname{claim}{Claim}{Claims}

\newtheorem*{question*}{Question}

\theoremstyle{definition}

\newtheorem*{definition*}{Definition}

\theoremstyle{remark}
\newtheorem*{remark}{Remark}

\newcommand{\mb}{\mathbb}

\newcommand{\on}{\operatorname}

\renewcommand{\ge}{\geqslant}

\usepackage{bm}

\allowdisplaybreaks

\title{On infinite sets with no $3$ on a line}

\author{Moe Putterman}
\address{OpenAI}
\email{mputt@openai.com}

\author{Mehtaab Sawhney}
\address{OpenAI and Columbia University}
\email{msawhney@openai.com}

\author{Gregory Valiant}
\address{OpenAI and Stanford University}
\email{valiant@openai.com}

\begin{document}

\begin{abstract}
We give a construction of an infinite set of points $A$ in $\mb{R}^2$ such that any subset $P\subseteq A$ has a constant density subset $P'$ with no three points collinear and yet $A$ cannot be separated into finitely many subsets such that each subset has no three points collinear. This provides a new proof of a question of Erd\H{o}s, Ne\v{s}et\v{r}il, and R\"{o}dl. The construction was generated by an internal model at OpenAI.
\end{abstract}

\maketitle

\section{Introduction}

\begin{theorem}\label{thm:main}
There exists an infinite set $A\subseteq \mb{R}^{2}$ with the following properties: 
\begin{itemize}
\item For any subset $P$ of $n$ points, there exists a subset $P'$ of size at least $n/2$ such that no three points in $P'$ are collinear.
\item $A$ cannot be split into finitely many sets $A_1 \cup \cdots \cup A_m$ such that each $A_i$ contains no three collinear points.
\end{itemize}
\end{theorem}
This resolves a question of Erd\H{o}s, Ne\v{s}et\v{r}il, and R\"{o}dl \cite{Er92b}. After sharing our proof of Theorem~\ref{thm:main} with R\"{o}dl, he noted in personal communication that Theorem~\ref{thm:main} can be deduced from work of \cite[Theorem~1.7]{ReiRodSal24} (noting that the only collinear triples in $[3]^{n}$ correspond to $3$-term arithmetic progressions and then a generic projection provides the result for Theorem~\ref{thm:main}). Due to the comparative simplicity of our proof, we record it here.

We further remark that one can pose a finitary version of the question by considering a set $A$ of $k$ points such that every subset satisfies the first item in Theorem~\ref{thm:main}, and asking how many sets $A_i$ are required to partition $A$ into pieces such that each piece has no three on a line. The construction here gives that $\Omega(\log k/\log\log k)$ sets may be required; determining the true dependence appears to be an interesting question. Note that an $O(\log k)$ upper bound follows in the finitary question via iteratively applying the first condition. (Our proof also gives a stronger quantitative dependence compared with \cite{ReiRodSal24} as we avoid the use of the density Hales--Jewett theorem.)

This theorem can be viewed as a geometric question of ``Pisier'' type. We refer the reader to Ne\v{s}et\v{r}il, R\"{o}dl and Sales \cite{NRS24} and Reiher, R\"{o}dl and Sales \cite{ReiRodSal24} for context, background and recent results regarding such questions.

\subsection{Comment on use of AI}
The construction and proof were generated by a model internal to OpenAI. The human authors digested the proof and have presented it in a human readable form (and modified it for clarity and elegance).

\section{Proof of Theorem~\ref{thm:main}}

Let $(t_i)_{i\in \mb{N}}$ be a sequence of real numbers that are algebraically independent.\footnote{Such a set may be constructed greedily; note that given $t_1,\ldots,t_k$ there are only countably many values $t_{k+1}^{\ast}$ such that there exists a nonzero polynomial $P$ with integer coefficients such that $P(t_1,\ldots,t_k,t_{k+1}^{\ast}) = 0$. Taking $t_{k+1}^{\ast}$ away from this countable set completes the proof.} We now construct $A$ to be the set of points indexed by $(i,j)\in \binom{\mb{N}}{2}$ such that $(i,j)\to (t_i + t_j, t_i^2 + t_it_j + t_j^2)$. We adopt the shorthand that $P_{i,j} = (t_i + t_j, t_i^2 + t_it_j + t_j^2)$.

We first classify the collinear triples. 
\begin{claim}\label{claim:colinear}
Let $i,j,k,l,m,n$ be distinct indices. Then  
\begin{itemize}
\item $P_{i,j}$, $P_{j,k}$, $P_{i,k}$ are collinear
\item $P_{i,j}$, $P_{i,k}$, $P_{i,\ell}$ are not collinear
\item $P_{i,j}$, $P_{j,k}$, $P_{k,\ell}$ are not collinear
\item $P_{i,j}$, $P_{i,k}$, $P_{\ell,m}$ are not collinear
\item $P_{i,j}$, $P_{k,\ell}$, $P_{m,n}$ are not collinear
\end{itemize}
\end{claim}
\begin{proof}
Recall that three points $(x_i,y_i)\in \mb{R}^{2}$ are collinear if and only if
\[
\on{det}\begin{pmatrix}
1 & x_1 & y_1\\
1 & x_2 & y_2\\
1 & x_3 & y_3
\end{pmatrix}=0.
\]

For the first part of Claim~\ref{claim:colinear}, we have that 
\[
\on{det}\begin{pmatrix}
1 & t_i+t_j & t_i^2+t_it_j+t_j^2\\
1 & t_i + t_k & t_i^2+t_it_k+t_k^2\\
1 & t_j+ t_k & t_j^2+t_jt_k+t_k^2
\end{pmatrix} = \on{det}\begin{pmatrix}
1 & t_i+t_j & t_i^2+t_it_j + t_j^2\\
0&  t_k-t_j & (t_k-t_j)(t_i+t_j+t_k)\\
0 & t_k-t_i & (t_k-t_i)(t_i+t_j+t_k)
\end{pmatrix}=0
\]
as desired. 

For the second claim, taking $P_{i,j} = (x_1,y_1)$, $P_{i,k} = (x_2,y_2)$ and $P_{i,\ell} = (x_3,y_3)$ note that 
\[
\on{det}\begin{pmatrix}
1 & t_i+ t_j & t_i^2+t_it_j+t_j^2\\
1 & t_i + t_k & t_i^2+t_it_k+t_k^2\\
1 & t_i+ t_\ell & t_i^2+t_it_\ell+t_\ell^2
\end{pmatrix}\neq 0.
\]
This is seen via examining the coefficient of $t_{\ell}^2t_k$ which is nonzero, since the only contribution to the determinant comes from the main diagonal in the Laplace expansion. For the third bullet, $P_{i,j} = (x_1,y_1)$, $P_{j,k} = (x_2,y_2)$ and $P_{k,\ell} = (x_3,y_3)$ we see that coefficient of $t_{\ell}^2t_k$ is nonzero via an analogous argument with the main diagonal. For the fourth bullet, $P_{i,j} = (x_1,y_1)$, $P_{i,k} = (x_2,y_2)$ and $P_{\ell,m} = (x_3,y_3)$; we see that coefficient of $t_{\ell}^2t_k$ is nonzero analogously. Finally for the fifth bullet with $P_{i,j} = (x_1,y_1)$, $P_{k,\ell} = (x_2,y_2)$ and $P_{m,n} = (x_3,y_3)$ we have that the coefficient of $t_m^2t_i$ is nonzero analogously. By the algebraic independence of $t_i$, each of the determinants is nonzero, as desired. 
\end{proof}
\begin{remark}
The model proceeded slightly differently here and instead more directly computes the determinant in one case and otherwise exhibits a valuation which is nonzero in the remaining cases.
\end{remark}

Given Claim~\ref{claim:colinear}, we complete our proof by arguing that the constructed $A$ satisfies both properties of Theorem~\ref{thm:main}.  For the first property, fix any set of points $P\subseteq A$; this point set can be naturally identified with a graph $G$ having $|P|$ edges. Since any graph $G$ with $|P|$ edges has a bipartite subgraph $G'$ with $\ge |P|/2$ edges, we have the desired first property.

For the second property, suppose that $A$ could be partitioned into $m$ sets, each of which has no three collinear points.  By Claim~\ref{claim:colinear}, this would correspond to an  $m$-coloring of $K_{\mb{N}}$ with no monochromatic triangle; this contradicts Ramsey's theorem, and in particular the finiteness of $R(3,\cdots,3)$.

\section*{Acknowledgments}
We thank Noga Alon and Vojta R\"{o}dl for useful comments. This research was conducted during the period MS served as a Clay Research Fellow.

\end{document}